\documentclass {amsart}
\usepackage{amssymb,amsthm,color, amscd, mathtools}
\begin{document} 
\title{\bf The Torsion Ideal Lattice of the Endomorphism Ring of an Abelian $p$--Group }

\author {Phill Schultz}
\address[phill.schultz@uwa.edu.au]{The University of Western Australia}

\theoremstyle{plain}
\newtheorem{theorem}{Theorem}[section]

\newtheorem{corollary}[theorem]{Corollary}
\newtheorem{lemma}[theorem]{Lemma}
\newtheorem{proposition}[theorem]{Proposition}
\newtheorem{conjecture}[theorem]{Conjecture}
\newtheorem{hypothesis}[theorem]{Hypothesis}
\newtheorem{condition}[theorem]{Condition}
\newtheorem{fact}[theorem]{Fact}
\newtheorem{problem}[theorem]{Problem}

\theoremstyle{definition}
\newtheorem{definition}[theorem]{Definition}
\newtheorem{notation}[theorem]{Notation}

\theoremstyle{remark}
\newtheorem{remark} [theorem]{Remark}
\newtheorem{remarks}[theorem]{Remarks}
\newtheorem{example}[theorem]{Example}
\newtheorem{examples}[theorem]{Examples}

\renewcommand{\leq}{\leqslant}
\renewcommand{\geq}{\geqslant}
\newcommand{\Aut}{\mathop{\mathrm{Aut}}\nolimits}
\newcommand{\End}{\mathop{\mathrm{End}}\nolimits}
\newcommand{\Ker}{\mathop{\mathrm{Ker}}\nolimits}
\newcommand{\Hom}{\mathop{\mathrm{Hom}}\nolimits}
\newcommand{\finind }{ ^{<\omega}}

\newcommand{\PR}{\begin\parbreak\end\parbreak\begin{proof}\parbreak\end{proof}\parbreak}
\newcommand{\LE}{\begin\parbreak\end\parbreak\begin{lemma}\parbreak\end{lemma}\parbreak}
\newcommand{\PROP}{\begin\parbreak\end\parbreak\begin{proposition}\parbreak\end{pro[osition}\parbreak}

\newcommand\bbQ{{\mathbb{Q}}}
\newcommand\bbZ{{\mathbb{Z}}}
\newcommand\bbN{{\mathbb{N}}}
\newcommand\bbT{{\mathbb{T}}}
\newcommand{\bbP}{\mathbb{P}}
\newcommand{\bbS}{\mathbb{S}}
\newcommand{\bbSN}{\mathbb{S}\mathbb{N}}
\newcommand{\bbF}{\mathbb{F}}

\newcommand{\boldi}{\mathbf{\iota}}
\newcommand{\boldn}{\mathbf{n}}
\newcommand{\boldb}{\mathbf{b}}
\newcommand{\boldc}{\mathbf{c}}
\newcommand{\bolde}{\mathbf{e}}
\newcommand{\bolda}{\mathbf{a}}
\newcommand{\boldd}{\mathbf{d}}
\newcommand{\bi}{\mathbf{\iota}}
\newcommand{\boldr}{\mathbf{r}}
\newcommand{\boldm}{\mathbf{m}}
\newcommand{\boldx}{\mathbf{x}}
\newcommand{\boldj}{\mathbf{j}}

\newcommand{\boldrho}{\boldsymbol{\rho}}
\newcommand{\boldnu}{\boldsymbol{\nu}}
\newcommand{\boldt}{\mathbf{t}}
\newcommand{\boldu}{\mathbf{u}}
\newcommand{\boldk}{\mathbf{k}}
\newcommand{\boldell}{\boldsymbol{\ell}}
\newcommand{\boldsigma}{\boldsymbol{\sigma}}

\newcommand{\boldtau}{\boldsymbol{\tau}}
\newcommand{\boldmu}{\boldsymbol{\mu}}

\newcommand{\calQ}{{\mathcal Q}}
\newcommand{\calP}{\mathcal{P}}
\newcommand{\calC}{\mathcal{C}}
\newcommand{\calD}{\mathcal{D}}
\newcommand{\calA}{\mathcal{A}}
\newcommand{\calS}{\mathcal{S}}
\newcommand{\calB}{\mathcal{B}}
\newcommand{\calX}{\mathcal{X}}
\newcommand{\calY}{\mathcal{Y}}
\newcommand{\calL}{\mathcal{L}}
\newcommand{\calM}{\mathcal{M}}
\newcommand{\calG}{\mathcal{G}}
\newcommand{\calT}{\mathcal{T}}
\newcommand{\calI}{\mathcal{I}\kern-1pt\textit nd}
\newcommand{\calW}{\mathcal{W}}
\newcommand{\calJ}{\mathcal{J}}
\newcommand{\calE}{\mathcal{E}}
\newcommand{\calK}{\mathcal{K}}
\newcommand{\calF}{\mathcal{F}}
\newcommand{\calV}{\mathcal{V}}
\newcommand{\calU}{\mathcal{U}}
\newcommand{\calN}{\mathcal{N}}
\newcommand{\calZ}{\mathcal{Z}}
\newcommand{\calLN}{\mathcal{LN}}
\newcommand{\calR}{\mathcal{R}}

\newcommand{\calId}{\mathcal{I}\kern-1pt\textit d}
\newcommand{\It}{\mathcal I\boldt}
\newcommand{\calFI}{\mathcal FI}

\newcommand{\Iinv}{\textbf{I-inv}}
\newcommand{\IM}{\mathtt{Im}}
\newcommand{\KER}{\mathtt{Ker}}

 \newcommand{\up}{^{(\xi)}}
 \newcommand{\LEM}{\begin{lemma}}
 \newcommand{\ov}{\overline}
\newcommand{\rank}{\operatorname{rank}}
\newcommand{\order}{\operatorname{order}}
\renewcommand{\Im}{\operatorname{Im}}
\newcommand{\type}{\operatorname{type}}
\newcommand{\lcm}{\operatorname{lcm}}
\newcommand{\trace}{\operatorname{trace}}
\newcommand{\Rad}{\operatorname{Rad}}
\newcommand{\ord}{\operatorname{order}}
\newcommand{\Jon}{\operatorname{Jon}}
\newcommand{\Sym}{\operatorname{Sym}}
\newcommand{\Ext}{\operatorname{Ext}}
\newcommand{\ACD}{\operatorname{ACD}}
\newcommand{\im}{\operatorname{im}}
\newcommand{\height}{\operatorname{height}}
\newcommand{\stem}{\operatorname{stem}}
\newcommand{\gap}{\operatorname{gap}}
\newcommand{\expp}{\operatorname{exp}}
\newcommand{\h}{\operatorname{height}}
\newcommand{\Ann}{\operatorname{Ann}}
\newcommand{\length}{\operatorname{length}}
\newcommand{\branch}{\operatorname{branch}}
\newcommand{\grad}{\operatorname{grad}}
\newcommand{\lift}{\operatorname{lift}}
\newcommand{\coexp}{\operatorname{co--exp}}
\newcommand{\Ulm}{\operatorname{Ulm}}
\newcommand{\dual}{\operatorname{dual}}
\newcommand{\di}{\operatorname{\dagger-{\rm inv}}}

\newcommand{\ind}{\operatorname{ind}}
\newcommand{\Range}{\operatorname{Range}}

\newcommand{\rot}{\operatorname{root}}
\newcommand{\Mono}{\operatorname{Mono}}
\newcommand{\Iso}{\operatorname{Iso}}
\newcommand{\divi}{\operatorname{div}}
\newcommand{\GL}{\operatorname{GL}}\
\newcommand{\Mod}{\operatorname{Mod-}}
\newcommand{\depth}{\operatorname{depth}}
\newcommand{\tr}{\operatorname{trace}}
\newcommand{\lat}{\operatorname{lat}}

\newcommand{\la}{\langle} 
\newcommand{\ra}{\rangle}      

\newcommand{\Sub}{\texttt{Subgroups}}
\newcommand{\QSub}{\texttt{Quasi-Subgroups}}
\newcommand{\Bases}{\texttt{Bases}}
\newcommand{\bases}{\texttt{Bases}}
\newcommand{\QBases}{ \texttt{Quasi-Bases}}
\newcommand{\Subgroup}{\texttt{Subgroup}}

\newcommand{\st}[1]{\la{#1}\ra_*}

\begin{abstract}  The    lattice of ideals of the torsion ideal of the endomorphism ring  of an abelian $p$--group is classified by   a system of cardinal invariants.
\end{abstract}
\noindent
 \maketitle

\section{Introduction}

The purpose of this paper is to classify the lattice of ideals of the torsion ideal   of the endomorphism ring of  a reduced abelian $p$--group. Let   $G$ be an unbounded reduced abelian $p$--group and $\boldt$ its ring of endomorphisms of finite order.   I show in Theorem \ref{t-ind} that    the lattice of ideals  of $\boldt$ is isomorphic to  the lattice under  the pointwise ordering of pairs $(\boldsigma, \boldm)$ where $\boldsigma$ is a $G$--admissible indicator with finite entries and $\boldm$ is a decreasing sequence of infinite cardinals. 
 This result generalises the   classification of the lattice of ideals of the endomorphism ring of a  bounded abelian $p$--group described in \cite{AvS} and \cite{ASZ}.
  
   After   recalling basic properties of $p$--groups in \S2,
in   \S3   I describe the known properties of the endomorphism ring   of $G$ and its torsion ideal $\boldt$.  In \S4 I construct a Galois Correspondence between the $\boldt$--invariant subgroups of $G$ and the ideals of $\boldt$. In \S5 I define the lattice of ranked  admissible indicators of $\boldt$. Finally in \S6 I use   ranked  admissible indicators to describe the lattice of ideals of $\boldt$. 
  \section{ Notation} 
Except when explicitly stated otherwise,  the noun \lq group\rq\  and the symbol $G$ signify  a reduced abelian $p$--group. Function names are written on the right of their arguments.
Otherwise, the notation  throughout this paper is standard, as found for example in \cite{Fuchs}.  
In particular,
 
\begin{itemize}\item  $\bbN$ denotes the poset of natural numbers including 0, $\bbN^+$ the positive integers, and for $n\in\bbN^+, \ [n]=\{1,\dots,n \}$;

\item For $n\in \bbN^+,\ \bbZ(p^n)$ is the cyclic $p$--group of order  $p^n$;

\item For  $a\in G$, the \textit{exponent  of $a$},\ $\exp(a)=\min\{k\in\bbN\colon p^ka=0\}$; $G$ is bounded
if $ \sup\{\exp(a)\colon a\in G\}$ is finite and if so, this supremum is denoted $\exp(G)$;

\item  $G[p^k]=\{a\in G\colon \exp(a)\leq k \}$; in particular, $G[p]$ is  the socle of $G$;
\item For any ordinal $\kappa,\  p^\kappa G[p]$ denotes $(p^\kappa G)[p]$;
\item For a subgroup $H$ of $G$, denoted $H\leq G,\ \rank(H)$ is the cardinality of a maximum linearly independent subset of $H$;
\item $\calE=\calE(G)$ is  the endomorphism ring of $G$ and $\boldt=\boldt(G)$ is the torsion ideal of $\calE$;  
 
\item $\bbF_p$ is the field of $p$ elements, so $G[p]$ is a $\bbF_p$--space of dimension $\rank(G[p])$;
\item For $a,\ b$ in a lattice $\calL,\ a\wedge b$ is the greatest lower  bound and  $a\vee b$ the  least upper bound of $a$   and $b$. 
\end{itemize}
The    definitions and main properties of basic subgroups,  heights, length, Ulm invariants  and indicators of a group and   separable  groups can be found in \cite[Chapters 5, 10 and 11]{Fuchs}. 
\subsection{$\Sigma$--cyclic groups} A  \textit{$\Sigma$--cyclic} group is  a direct sum of cyclic groups. Each     $\Sigma$--cyclic group $B$ has a \textit{standard presentation} of the form $B=\bigoplus_{i\in S} B_i$ where $S=[s],s\in\bbN^+$, or $S=\bbN^+$ and  for all $i,\,j\in S$

 \begin{enumerate}

\item $B_i\cong \bigoplus_{m_i}\bbZ(p^{n_i}),\ n_i\in\bbN^+$ and $m_i$  is  a non--zero cardinal;
 \item $i<j$ implies $ {n_i} <n_j$.
 
 \end{enumerate} 
 
 A $\Sigma$--cyclic  group $B$ is denoted by $B(\boldn,\,\boldm)$, where $\boldn=(n_i\colon i\in S)$ and $\boldm=(m_i\colon i\in S)$.
 Thus $\exp(B_i)=n_i$ and $\rank(B_i)=m_i$.
 
It is shown in \cite[Chapter 5, \S 5]{Fuchs} that every   group $G$   has  a  \textit{basic subgroup} $B(\boldn,\,\boldm)$, unique up to isomorphism,  satisfying the following properties:
\begin{enumerate}  \item  $B$ is a $\Sigma$--cyclic  group, pure in $G$, and for all $i\in S,\ \bigoplus_{j\leq i}B_j$ is a maximal pure $p^{n_i}$--bounded subgroup of $G$;

\item   For all $i\in S,\ G=\bigoplus_{j\leq i}B_j\oplus\left(\bigoplus_{j>i}B_j +p^iG\right)$; 
\item    $G/B$ is divisible;  thus every $f\in\calE$ is determined by its value on $B$  i.e., if $f,\ g\in\calE$ satisfy  $g|_B=f|_B$ then $f=g$.  In particular, if $f$ is zero on $B$, then $f=0$;
\item   If $A$ is a $\Sigma$--cyclic pure subgroup of $G$, then $A$  is a direct summand of some basic subgroup of $G$.
\end{enumerate}

 \subsection{Ulm invariants} (\cite[Chapter 10, \S1]{Fuchs}) Let $G$ be a group of length $\lambda$.
  For all $\kappa<\lambda$, the \textit{$\kappa$--Ulm invariant}, $u_\kappa=\rank(p^\kappa G[p]/p^{\kappa+1}G[p])$.
    
 Let  $\boldu(G)$ be the sequence $(u_\kappa\colon \kappa<\lambda\colon u_{\kappa}\ne 0)$ of non--zero Ulm invariants of $G$. 
   
 \subsection{Indicators}  The following notation and properties of indicators are found in \cite[Chapter 10, \S1]{Fuchs}:
 \begin{enumerate} 
\item An \textit{indicator} $\boldsigma=(\sigma_i\colon i\in\bbN)$ is a   strictly  increasing  finite sequence of ordinals $\sigma_i,\ i\leq n$, followed by a countably infinite sequence of symbols $\infty$, or a strictly increasing  sequence of ordinals.
In the former case, we abbreviate the indicator $(\sigma_0,\dots,\sigma_n,\infty,\dots,\infty)$ by $(\sigma_0,\dots,\sigma_n,\infty)$. 
The set of  indicators   is denoted  $\calI$.

\item   An indicator $\boldsigma$ has \textit{  length $n$} in the former case and \textit{infinite length} in the latter case.   An indicator  $\boldsigma$ has \textit{finite entries} if  each ordinal entry $\sigma_i$ is finite. The \textit{zero indicator}, denoted $(\infty)$, has length 0 and is considered to have finite entries. 

\item $\calI$ is ordered pointwise:   $\boldsigma$ \textit{precedes} $\boldtau$, denoted $\boldsigma\preceq\boldtau$ if and only if for all $i<\omega,\ \sigma_i\leq\tau_i$, where we regard an ordinal $\sigma_i<\infty$ and $\infty\leq\infty$. 
 
\item $\calI$  is a well--ordered set with minimum $\boldsigma^{min}=(i\colon i<\omega)$ and maximum $\boldsigma^{max}=(\infty)$.
 If $\calS\subseteq\calI$  then $\bigwedge\calS=(\sigma_i)$ where $\sigma_i=\min_{\boldtau\in\calS}\tau_i$ and $\bigvee\calS= (\rho_i)$ where  $\rho_i=\sup_{\boldtau\in\calS}\tau_i$, so $\calI$  is a complete lattice. 

 \item The set of    indicators   with finite entries is denoted  $\calI^{<\omega}$.
$\calI^{<\omega}$ is a sublattice of  $\calI$, closed under arbitrary meets  and finite  joins.

\item   $\boldsigma\in\calI$    \textit{has a gap at $i$} if $\sigma_i+1<\sigma_{i+1}<\infty$. 	 
\end{enumerate}
Indicators are used to define certain invariants of groups. Let   $G$ be a group and   $\boldsigma\in\calI$. $\boldsigma$  is \textit{$G$--admissible} if 

\begin{enumerate}\item   For all $i<\omega$, if $\boldsigma$ has a gap at $i$ and $\sigma_i=\kappa$, then the   Ulm invariant  $u_\kappa\ne 0$;  
\item If $\exp(G)=e$ then $\boldsigma$ has length $ e$.
\end{enumerate}
For example, for any group $G$, $(i\colon i\in[0,n],\infty),\, (i\colon i<\omega)$ and $(\infty)$ are   $G$--admissible.

 Let $\calI(G)$ be the set of  $G$--admissible indicators and $\calI^{<\omega}(G)=\calI(G)\cap  \calI^{<\omega}$. The following properties of $\calI(G)$ are either proved in \cite[Chapter 10]{Fuchs} or are straightforward consequences of the definitions:

\begin{enumerate} \item If $G$ is $\Sigma$--cyclic, then every $\boldsigma\in\calI(G)$ has finite entries. Conversely, if $G$ has basic subgroup $B$ and $\boldsigma\in \calI^{<\omega}(G)$, then $\boldsigma\in \calI(B)$.
\item Let $a\in G$ with $\exp(a)=n+1$. Let  $\sigma_i=\height(p^ia)$. The \textit{indicator of $a$} is the $G$--admissible indicator $\ind(a)=(\sigma_i\colon i\in[0,n],\infty)$.  
 Conversely, if $\boldsigma\in\calI(G)$   is  finite   of length $n$,   then there exists  $a\in G$  with $\exp(a)=n$ such that $\boldsigma=\ind(a)$;  

\item   For all $\boldsigma\in\calI(G)$, let $G(\boldsigma)=   \{a\in G\colon \sigma \preceq\ind(a) \}$. 
$G(\boldsigma)$ is a fully invariant subgroup of $G$, called an  \textit{indicator subgroup}. 

\item   Let $\boldsigma,\ \boldtau\in \calI(G)$. Then $\boldsigma\preceq\boldtau$ if and only if $G(\boldtau)\leq G(\boldsigma)$.

\item   Let $a\in G$ and $ f\in \calE$. Then $\ind(a)\preceq\ind(af)$.
$G$ is \textit{fully transitive} if for all $a,\,b\in G,\   \ind(a)\preceq\ind(b)$ implies there exists $f\in\calE$ such that $b=af$. 
Simply presented and separable groups are fully transitive.
  \end{enumerate}
  
\begin{proposition}\label{ndsub} Let  $G$ be a group and let $\boldsigma,\ \boldtau\in\calI(G)$
\begin{enumerate}
\item  For all $i<\omega,\   
  \sigma_i=\min\{\height(p^ia)\colon a\in G(\boldsigma)\}$.
\item   $\boldsigma=\boldtau$ if and only if $G(\boldtau)= G(\boldsigma)$.

  \item $\calI(G)$ is a   lattice anti--isomorphic to $G(\calI)$;
 \item $\calI^{<\omega}(G)$ is a sublattice anti--isomorphic to $G(\calI^{<\omega})$.
\end{enumerate}
\end{proposition}

\begin{proof} 

(1) Let $a\in G(\boldsigma)$. Then for all $i\leq\exp(a), \ \sigma_i\leq\height(p^ia)$. Conversely, if for all $i,\ \sigma_i\leq\height(p^ia)$ then $\sigma\preceq\ind(a)$ so $a\in G(\boldsigma)$. Hence
 for all $i<\omega,\   \min_{a\in G}\{\height(p^ia)\}=\sigma_i$.

(2) Necessity is obvious. Suppose  $G(\boldtau)=G(\boldsigma)$. Then by (1),  for all $i, \tau_i=\min\{\height(p^ia)\colon a\in G(\boldtau)\}=\min\{\height(p^ia)\colon a\in G(\boldsigma)\}=\sigma_i$.

(3)   It is straightforward to verify that if $S$ is a finite  set of $G$--admissible indicators, then $\bigvee S$ and $\bigwedge S$ are $G$--admissible, and $G(\bigvee S)$ and $G(\bigwedge S)$ are fully invariant in $G$. The map $\calI(G)\to G(\calI)$ defined by $\boldsigma\mapsto G(\boldsigma)$ reverses order, and by (2) it is a bijection.  Hence the map is a a lattice anti--tsomorphism.
   
(4) The anti--isomorphism  (3) maps $\calI^{<\omega}(G)$ to $G(\calI^{<\omega})$
\end{proof}
\begin{corollary}\label{indicat}
  Let $G$ be a group.   The mapping $\boldsigma\mapsto G(\boldsigma)$
   is a lattice  anti--isomorphism of   $\calI(G)$ to $G(\calI)$, mapping finite length indicators to bounded groups  and  $\calI^{<\omega}(G)$ to $G(\calI^{<\omega})$.
  \qed\end{corollary}
  
  \begin{lemma}\label{Lem6}  Let $\boldsigma=(\sigma_i)\in \calI$ and for each $n\in\bbN$ let $\boldsigma^{(n)}$  be the truncation $(\sigma_0,\dots,\sigma_n,\infty)$.
\begin{enumerate}\item $\boldsigma$ is $G$--admissible if and only if each 
 $\boldsigma^{(n)}$ is $G$--admissible;
 \item  $\boldsigma =\inf_{n\in\bbN} \{\boldsigma^{(n)}\}$.
 \end{enumerate}
 \end{lemma}
 
 \begin{proof} (1) $\boldsigma$ has a gap at $i$ if and only if whenever  $i<n,\ \boldsigma^{(n)}$ has a gap at $i$.
 
 (2) For all $n\in\bbN,\  \boldsigma \preceq  \boldsigma^{(n)} $. Conversely, if for all $n\in\bbN\  \boldtau \preceq  \boldsigma^{(n)}$,  then $\boldsigma\preceq\boldtau$.
 \end{proof}
 
 \begin{corollary}\label{Cor7} If $\boldsigma$ is $G$--admissible, then $G(\boldsigma)=\sum_{n\in\bbN} G(\boldsigma^{(n)})$.
 \qed\end{corollary}

         \begin{proposition}\label{sep} Let $G$ be unbounded with basic subgroup $B$.
       \begin{enumerate}
         \item $G/p^\omega G$ has a basic subgroup isomorphic to $B$;
        \item $\calI^{<\omega}(G)= \calI(G/p^\omega G)=\calI(B)$.
        \end{enumerate}
        \end{proposition}
        
        \begin{proof}   
 (1)    For all $n\in\bbN,\ G=\bigoplus_{i\leq n}B_i \oplus(\bigoplus_{i>n}B_i + p^nG)$. Hence for all $n\in\bbN,\ G/p^\omega G=\bigoplus_{i\leq n}B_i \oplus(\bigoplus_{i>n}B_i + p^nG)/p^\omega G$. Thus $G/p^\omega G$ contains for all $n<\omega,\   \bigoplus_{i\leq n}B_i + p^n(G/p^\omega G)$ so $G/p^\omega G$ has $B$ as a subgroup with divisible factor group.  Consequently, $G/p^\omega G$ has a basic subgroup isomorphic to $B$.

    (2)        
    $\boldsigma\in \calI^{<\omega}(G)$ if and only if $\boldsigma$ is a $G$--admissible sequence with finite entries if and only if $\boldsigma$ is a $B$--admissible sequence if and only if $\boldsigma$ is a $G/p^\omega G$--admissible sequence
         \end{proof}

\subsection{Separable Groups}  
A group of length $\leq\omega$,  i.e., either bounded or unbounded   with no elements of infinite height, is called \textit{separable}. 
    Let $G$ be a separable group with basic subgroup $B=B(\boldn,\boldm)=\bigoplus_{i<\omega}B_{n_i}$.  It is shown in \cite[Chapter 10, \S3]{Fuchs} that
 if $G$ is unbounded, the  $p$--adic  completion of $B$ is a subgroup of $\prod_{i<\omega}B_{n_i}$ whose  torsion subgroup    $\ov B$    is a $p$--group consisting of   all bounded sequences in   $\prod_{i<\omega}B_{n_i}$. $\ov B$ is   called the \textit{torsion completion   of $B$}. If $|B|=\mu$, then $|\ov B|=2^\mu$ and $\ov B/B$ is divisible of rank $2^\mu$.

Throughout this section, we represent elements of $\ov B$ as bounded sequences $\boldb=(b_i\colon i\in\bbN)\in\prod_{i\in\bbN}B_{n_i}$ where $b_i\in B_{n_i}$.

 Fuchs \cite[Chapter 10]{Fuchs} shows   that    the following properties are equivalent:
  
 \begin{enumerate} \item  $G$ is separable;
 \item   every finite subset of $G$ is contained in a finite direct summand;  
 \item  $B\leq G\leq \ov B$, $G$ is pure in $\ov B$ and $\ov B/G$ is divisible.
 
 \end{enumerate}
Consequently,  a separable   $G$ with basic subgroup $B$ is   simply presented if and only if  $G=B$.

On the other hand, separable groups are {fully transitive} (\cite[Chapter b10, Corollary 1.5]{Fuchs}).

Consequently, by Kaplansky's Theorem, \cite[Chapter 10, Theorem 2.2]{Fuchs}, every fully invariant subgroup $H=G(\boldsigma)$ for some $\boldsigma\in\calI(G)=\calI^{<\omega}(G)$.  
  
Countable separable groups are $\Sigma$--cyclic,   but    uncountable separable groups  are notoriously complicated; as Fuchs says,   \cite[page 316]{Fuchs},  no general structure theorem is available, and none is expected. 
In particular,    Corner \cite{Corn} shows that even in the simplest case when each $B_n\cong \bbZ(p^n)$ there exists a family of $2^{2^{\aleph_0}}$ non--isomorphic pure subgroups $G$  of $\ov B$ containing $B$. 

This result holds for any $|B|\leq 2^{\aleph_0}$;   for larger $|B|$,   Shelah has shown, \cite[p. 332]{Fuchs}, that for every infinite cardinal $\kappa$ there are $2^\kappa$ non--isomorphic separable groups   of cardinal $\kappa$ with the same basic subgroup.

\begin{lemma}\label{aresep} For every group $G$,
  $G/p^\omega G$ is separable.
 \end{lemma}
\begin{proof}   By way of contradiction, let $0\ne a\in G\setminus p^\omega G$ such that the height of $a+ p^\omega G$ in $G/p^\omega G$ is infinite.   
Then for all $n\in\bbN$ there exist $ b_n$  and $c_n$ in $G$ such that   $a-p^nb=c_n\in p^\omega G$. Hence there exist $d_n\in G$ such that $p^nd_n=c_n$, so  for all $n,\ a=p^n(b_n+d_n)\in p^n G$, so $a\in p^\omega G$, a contradiction.
 \end{proof}

 \section{The endomorphism ring of $G$}

Our current knowledge of the structure of the endomorphism ring $\calE$ of an abelian  $p$--group $G$ is presented in \cite[Chapter 4, \S 20]{KMT} which mainly draws on the results presented in \cite{Fuchs} and \cite{Pierce}.
For the particular case of separable $G$, see \cite{Corn} and \cite{String}.
These references  are mainly concerned with identifying the Jacobson radical $\calJ$ of $\calE$ and the factor ring $\calE/\calJ$. The  main structure theorem  states that  if $G$ has non--zero Ulm invariants $(u_\kappa\colon \kappa<\lambda)$  then $\calE$ is  the split extension of the Jacobson radical  $\calJ$ of $\calE$ by a subring of the  direct product $\prod_{i<\lambda}\calL(u_i)$ where $\calL(u_i)$ is the ring of linear transformations of the $\bbF_p$--space of dimension  $u_i$.  $\calJ$ is charatacterised  as the $p$--adic closure of the ideal of endomorphisms which raise  heights in the socle, or alternatively as the ideal of endomorphisms $f$ for which the restriction of $1-f$ to the socle  $G[p]$   is monic and epic.

The article \cite{AvS} contains an explicit description, in terms of cardinal invariants,  of the ideal lattice of $\calE$ in   the bounded case. Let $G=\bigoplus_{i\in [s]}B_i$ where $B_i$ is homocyclic of exponent $n_i$ and rank $m_i$. Then  $\calE$ is represented as the matrix ring $(\calE_{ij})$ where $\calE_{ij}=\Hom(B_i,B_j)$. The  main theorem \cite[Theorem 4.5]{AvS}  shows that an ideal $I$ of $\calE$  is a matrix ring $(I_{ij})$ where $I_{ij}$ is a subgroup of $\calE_{ij}$ 
of the form $I_{ij}=\calE_{i,j}(\boldk_{ij}, \boldmu_{ij})=\sum_{u\in[0.t]}p^{k_{ij}(u)}\calE_{ij}^{\mu_{ij}(u)}$. Here $\calE_{ij}^{\mu_{ij}(u)}=\{f\in \calE_{ij}\colon \rank(f)\leq\mu_{ij}(u)\},\ 
\boldk_{ij}=(k_{ij}(0),\dots, k_{ij}(t_{ij}))$ is an increasing  sequence of positive integers and $\boldmu_{ij}=(\mu_{ij}(0),\dots,\mu_{ij}(t_{ij}))$ a non--increasing sequence of cardinals. The sequences $\boldk_{ij}$  and $\boldmu_{ij}$ satisfy a system of inequalitites determined by the parameters $n_i$ and $m_i$ of the subgroups $B_i$.  
 
Thus the  ideals of $\calE$ are characterised by an $s\times s$ matrix whose terms are homomorphism groups parametrised by three variables: height, exponent and rank. 
This result is extended in \cite{AbS} to separable reduced $p$--groups by using row convergent $\omega\times\omega$ matrices $(I_{ij})$ with similar parameters. 
 
The aim of the rest of this paper is to classify the  ideals of  the torsion ideal $\boldt$ of $\calE$      for arbitrary  unbounded $G$. Instead of representing $\boldt$ by matrices $(\Hom (B_i,B_j)$ and its ideals by matrices $(I_{ij})$   as in \cite{AvS} and \cite{AbS}, we characterise    ideals    of    $\boldt$  by indicators and ranks.  When applied to bounded groups, this is a new, co--ordinate free version of the results of   \cite{AvS} and  \cite{AbS}.

\section{The torsion ideal of $\calE$}

Let $G$ be an unbounded group with  endomorphism ring $\calE$. It is well known (\cite[Chapter 1, Lemma 8.3]{Fuchs}) that $G$ is a $p$--adic module under the natural action of the $p$--adic integers and $\calE$ is its $p$--adic endomorphism ring.  Let  $\boldt$ be the set  all endomorphisms $f\in\calE$ of finite order. It is routine to check that $\boldt$ is an ideal  of $\calE$, called the
\textit{torsion ideal}. Since $\calE$ is a reduced mixed $p$--adic algebra,
 $\boldt$ is a reduced unbounded $p$--group.. 
Henceforth, to simplify the exposition, we shall consider   $\boldt$   either as  an abelian  $p$--group  or  as an   ideal of $\calE$, according to context.  

  In  the rest of this section,   $G$ is an unbounded group with basic subgroup $B(\boldn,\,\boldm)$ and   endomorphism ring $\calE$ with torsion ideal $\boldt$.
 
 \begin{lemma}\label{Ei1}  
If $f\in\calE$ has infinite $p$--height, then $f=0$.
\end{lemma}
 \begin{proof}    For all $n\in\bbN$, let $g_n\in\calE$ such that $f=p^ng_n$, and let $a\in G$. If $\exp(a)=n$ then   $af=p^nag_n=0$.
\end{proof}
\begin{corollary}\label{separ} $\boldt$ is a separable abelian $p$--group.
\end{corollary}
\begin{proof} If $f\in\boldt$, then $f$ has $p$-power order and finite height in $\calE$ so  finite height in $\boldt$.
\end{proof}

\subsection{The abelian group invariants of $\boldt$}
As a  separable $p$--group, $\boldt$ has admissible indicators which by Lemma \ref{Ei1} have only finite entries. Let $\calI(\boldt)$ be the lattice of $\boldt$--admissible indicators.

\begin{lemma}\label{preser}  Let  $G$ be an unbounded    group and    $f\in\boldt$. Then $f\in   p^n\boldt[p^i]$ if and only if $Gf\leq p^nG[p^i]$. 

\end{lemma}
 
\begin{proof} For all $a\in G,\ af\in p^nG[p^i]$.  Conversely, let   $a\in  p^nG[p^i]$. Since there exists a cyclic summand $\la b\ra$ of $G$ of   exponent $>n+i$, there exists $f\in p^n\boldt[p^i]$ such that $bf=a$.
\end{proof}
\begin{proposition}\label{ip}Let  $G$ be a group with endomorphism ring $\calE$ and torsion ideal $\boldt$. The following are equivalent:
 \begin{enumerate}
  \item   $G$  has a cyclic summand of exponent $n$;
   \item   $\boldt$ has an idempotent generated cyclic summand of exponent $n$;
 \item $\boldt$ has a cyclic summand of exponent $n$.

\end{enumerate}   \end{proposition} 
 \begin{proof} $(1)\Rightarrow (2)$ Let $\la a\ra$ be a cyclic summand of $G$ of exponent $n$.  Then the natural projection $f$ of $G$ onto $\la a\ra$ is an idempotent element of $\boldt$ such that $\la f\ra$ is a 
cyclic summand of $\boldt$ exponent $n$.

 $(2)\Rightarrow (3)$ is trivial.
 
 $ (3)\Rightarrow (1)$  
 By Lemma \ref{preser}, for all  $ n\in\bbN^+,\ \im \left(p^n\boldt[p]\right)=p^nG[p]$.  Recall that for any group $H,\ p^{n-1}H[p]\ne p^nH[p]$ if and only if $H$ has a cyclic summand of exponent $n$.
  
  Thus $\boldt$ has a cyclic summand of exponent $n$ implies $p^{n-1}\boldt[p]\ne p^n\boldt[p]$ so $p^{n-1}G[p]\ne p^n G [p]$ and hence   $G$  has a cyclic summand of exponent $n$.
  \end{proof}

  \begin{corollary} \label{cyc}\begin{enumerate}\item $p^k\boldt[p]/p^{k+1}\boldt]p]\ne 0$    if and only if $p^kG[p]/p^{k+1}G\ne 0$;
  \item Let $\boldsigma$ be an indicator with finite entries. Then $\boldsigma\in\calI^{<\omega}(G)$ if and only if $\boldsigma\in \calI(\boldt)$.
 \qed \end{enumerate}\end{corollary}  
 
  \begin{proposition}\label{basi} Let  $G$ be unbounded with basic subgroup $B=\bigoplus_{i<\omega} B_i=B(\boldn,\,\boldm)$.   For each $n<\omega$, let $\rank(G/\bigoplus_{i\in[n-1]}B_i)=\mu_n$.
  
  Let $E=\bigoplus_{i<\omega} E_i=B(\boldnu,\boldk)$ be a basic subgroup of $\boldt$. Then
\begin{enumerate}
 \item  $\boldn=\boldnu$
 
 \item  If $ m_i$ is finite, then $k_i= m_i\times \mu_i$;
 if $m_i$ is infinite, then $k_i= 2^{m_i}\times \mu_i$
  \end{enumerate}
 \end{proposition}
 \begin{proof}

(1)  This follows from Proposition \ref{ip}.

(2) Since $\bigoplus_{i\in[n-1]}B_i$ is the summand of $G$ consisting of   cyclic summands of exponent $<n,\ E_i\cong \Hom(B_i, G/\bigoplus_{i\in[n]}B_i)\cong \prod_{m_i}\Hom(\bbZ(p^{n_i}), G/\bigoplus_{i\in[n]}B_i)$.
If $m_i$ is finite, then $\rank(E_i)=m_i\times \mu_i$, while if $m_i$ is infinite, $\rank(E_i)= 2^{m_i}\times \mu_i$.
 \end{proof}
 
 The abelian group invariants of $\boldt$ follow immediately:
  \begin{theorem}\label{obvious}  Let $G$ be an  unbounded group with basic subgoup $B=B(\boldn,\boldm)$ and let $\boldt$ be the torsion ideal of the endomorphism ring $\calE$. Let $(\boldnu,\boldk)$ be the   sequence defined in Proposition \ref{basi}.
  
  Then $\boldt$ is a   separable   $p$--group with basic subgroup $B(\boldnu,\boldk)$ and the   factor ring $\calE/\boldt$ is a  reduced  torsion--free $p$--adic module.
\end{theorem}

\begin{proof} The result is clear if $G$ is bounded, so assume $G$ is unbounded.

Since $\calE$ is a $p$--adic algebra, $ \boldt$ is a $p$--group and $\calE/\boldt$ a torsion--free $p$--adic module. Since $G$ has an unbounded basic subgroup, $\boldt$     is also unbounded.
By Corollary \ref{separ}     $\boldt$   is a separable $p$--group.   By Proposition \ref{basi} $\boldt$ has a basic subgroup of   isomorphic to  $\bigoplus_i \End(B_i)$.

Since $\calE/\boldt$ is a    torsion--free $p$--adic module, it remains to show $\calE/\boldt$ is reduced.

For all $f\in \calE$, denote $f+\boldt\in\calE/\boldt$ by $\ov f$.  Suppose $\calE/\boldt$ has a divisible $p$--adic summand $D$.   
Let $0\ne\ov f\in D$. Then there exist an increasing sequence $(n_i\colon i<\omega,\,n_i\in\bbN)$ and a sequence $\{ {f_i}\colon i<\omega,\ f_i\in\calE\}$   such that $\ov f=p^{n_i}\ov{f_i} $. For all $i>1$,\ let $\exp(f-p^{n_i }{f_{i }})=m_i$. Then $p^{m_i}(f-p^{n_i}f_i)=0$. Since $\calE/\boldt$ is torsion--free, for all $i<\omega,\ f=p^{n_i}f_i$; so $f$ is divsible in $\calE$, a contradiction. Hence $\calE/\boldt$ is reduced. 
\end{proof}

\begin{corollary}\label{4.3} Under the hypotheses of Theorem \ref{obvious}, let $E=\bigoplus_{i<\omega}E_i$.

\begin{enumerate}\item $E\leq\boldt\leq\ov{E}$ where $\boldt$ is 
  a pure subring of  $\ov{E}$ with divisible factor group;
  \item If $G$ is  $\Sigma$--cyclic, then $\boldt=\ov{E}$.
 \end{enumerate} \end{corollary}
  \begin{proof} (1) Section 2 describes the structure of separable  $p$--groups with   basic subgroup
  $E$.
  
  (2)    Since every $f\in \End(B)$  extends uniquely to $\End(\ov B),\ \boldt=\ov E $. 
  \end{proof}
  
  Although we cannot classify the groups $H$ between $E$ and $\ov E$, we can at least identify those which are rings:

   \begin{proposition}\label{:} Let $G$   be an unbounded group and let $E=\bigoplus_i E_i$ be a basic subgroup of $\boldt$. Let
    $E\leq H\leq\ov E$. Then $H$ is a subring of $\ov E$ if and only if whenever $(a_i)$ and $(b_i)\in H,\ (a_ib_i)\in H$.   
  \end{proposition}
  \begin{proof} Suppose $H$ is a subring of $\ov E$ containing $E$ and let $a=(a_i)$ and $b=(b_i)\in H$. Then for all $i<\omega,\ a_i,\ b_i\in E_i$ so  $a_ib_i\in E_i$ and hence $H$ is closed under addtion. 
  
  Conversely, since the distributive properties are defined pointwise, closure under multiplication  suffices   to ensure  $H$ is a ring. If $(a_i)$ and $(b_i)\in H$   imply $(a_ib_i)\in H$, then $H$ is closed under multiplication, and hence is a ring.
    \end{proof}

\section{The action of $\boldt$ on $G$}

\begin{notation} 
A subgroup $H$ of $G$ is \textit{ $\boldt$--invariant} if for all $f\in \boldt,\ Hf\leq H$. 

Denote by $\boldt$-inv$(G)$ the set of  $\boldt$--invariant subgroups of $G$, and by $\calId(\boldt)$ the set of ideals of $\boldt$.

\end{notation}  

 We have seen that for all $G$--admissible indicators $\boldsigma,\ G(\boldsigma)$ is a fully invariant and hence $\boldt$--invariant subgroup of $G$. I now show that conversely, for every $\boldt$--invariant subgroup $H$ of $G,\ H=G(\boldsigma)$ for a unique $\boldt$--admissible $\boldsigma$. 
 
 \begin{proposition}\label{4.11} Let $H\in  \boldt$-inv$(G)$. For all $i<\omega$, let $\sigma_i=\min\{\height(p^ia)\colon a\in H\}$. Then $\boldsigma$ is $\boldt$--admissible and $H=G(\boldsigma)$.
 \end{proposition}
 
 \begin{proof} For all $a\in H,\ \height(p^ia)\geq\sigma_i$, so $H\leq G(\boldsigma)$.
 
 Conversely, let $b\in G(\boldsigma)$ have exponent $n$. For all $i<\omega,\ \min\{\height(p^ia)\colon a\in H\}$ is realised  for  $a\in H[p^{i+1}]$, so $\boldsigma$ is $\boldt$--admissible and  for all $i<\omega$ there exist $a_i\in H[p^{i+1}]$ with $\height(p^ia_i)=\sigma_i$. Let $a=\sum_{i<n}a_i$. Then $a\in H$ and for all $i<\omega,\ \height(p^ia)=\sigma_i\leq \height(p^ib)$. Since $\{a_i\colon i\in[0,n]\}$ is contained in a finite summand of $G$, there exists $f\in\boldt$ such that $af=b$. Hence $b\in H$.
  \end{proof}
\begin{proposition}\label{fik} $\boldt$-inv$(G)$ and $\calId(\boldt)$ are complete lattices under inclusion. 
\end{proposition}
\begin{proof} Let $F\subseteq \boldt$-inv$(G)$ and $f\in \boldt$. Then $(\bigcap F)f =(\bigcap Ff)\leq \bigcap F$ so $\boldt$-inv$(G)$ is closed under meets.
Similarly $(\sum F)f= (\sum Ff)\leq \sum F$  so $\boldt$-inv$(G)$ is closed under joins.

Let $J\subseteq \calId(\boldt)$. Then $\bigcap J$ and $\sum J$ are   ideals of $\boldt$.  
\end{proof}

I now define ideals of $\boldt$ analogous to the indicator sugroups of $G$.
Recall that for $\boldsigma\in \boldt(\calI),\ \boldt(\boldsigma)=\{f\in\boldt\colon \boldsigma\preceq \ind(f)\}$.  
\begin{proposition}\label{4.0} \begin{enumerate} 
\item For all $\boldsigma\in \boldt(\calI),\ \boldt(\boldsigma)\in\calId(\boldt)$.   

\item $\boldsigma\preceq\boldtau$ if and only if $\boldt(\boldtau)\leq \boldt(\boldsigma)$.

\item The set $\{\boldt(\boldsigma)\colon \boldsigma\in \boldt(\calI)\}$ is a sublattice of $\calId(\boldt)$.
\end{enumerate}\end{proposition}
\begin{proof}  
(1) For all $\boldsigma\in \boldt(\calI),\ \boldt(\boldsigma)$ is closed under addition. Since endomorphisms do not decrease heights, $\boldt(\boldsigma)$ is also closed under left and right multiplication from $\boldt$.

(2) and (3) These are  group theoretic  properties, which hold by   Proposition \ref{ndsub}.
\end{proof}

  I now construct  a Galois Connection (\cite[Chapter 5, \S 8]{Birkhoff},\ \cite{AbS}) between  the lattices   $\boldt$-inv$(G)$    and  $\calId(\boldt)$.

\begin{notation}    For all $H\in \boldt$-inv$(G)$,  let 
  $  \IM^{-1} (H) =\{f\in\boldt\colon \im(f)\leq H\}$.
  
 For all $I\in \calId(\boldt)$,  let  $\IM(I)=\la   Gf\colon\ f\in I\ra$.
 \end{notation}  
\begin{lemma}\label{dire} \begin{enumerate}\item For all $H\in\boldt$-inv$(G),
\ \IM^{-1}(H)\in \calId(\boldt)$; 

\item  For all $I\in\calId(\boldt),\  \IM(I) \in\boldt$-inv$(G)$.
\end{enumerate} \end{lemma}
 
 \begin{proof} (1) Let $f,\,g\in \IM^{-1}(H)$. Then $f-g\in \IM^{-1}(H)$. Let $h\in \boldt$. Then $hf\in \IM^{-1}(H)$. Since $H$ is   $\boldt$--invariant, $fh\in \IM^{-1}(H)$. Hence $\IM^{-1}(H)$ is an ideal of $\boldt$.

(2) Let  $a=bf\in \IM(I)$. Then for all $g\in \boldt,\ ag=bfg\in\IM(I)$. Hence $ \IM(I)$ is $\boldt$--invariant subgroup of $G$. 
 \end{proof}

 \begin{proposition}\label{Gal} With the notation above, 

\begin{enumerate}\item  

   The mappings $  \IM^{-1} \colon \boldt$-inv$(G)\to\calId (\boldt)$ and  $  \IM\colon \calId (\boldt)\to \boldt$-inv$(G)$ preserve  order,  meets   and joins.

 \item For all $H\in \boldt$-inv$(G)$ and all   $I\in \calId(\boldt)$,

\begin{enumerate}  
\item $ \IM\,\IM^{-1}\,(H)\leq H$ and   $ \IM^{-1}\,\IM (I) \leq I$;
\item $\IM^{-1}\,\IM\,\IM^{-1}\,(H) = \IM^{-1}(H) $ and $\IM\,\IM^{-1}\,\IM(I) = \IM(I)$. 
 \end{enumerate}
\end{enumerate}
\end{proposition}
\begin{proof} (1)   The definitions imply that both mappings preserve inclusions.

Let $H,\,K\in\boldt$-inv$(G)$. Then $H+K\in\boldt$-inv$(G)$ and  $\IM^{-1}(H+K)=\IM^{-1}H +\IM^{-1}K$.
Also  $H\cap K\in\boldt$-inv$(G)$ and  $\IM^{-1}(H\cap K)=\IM^{-1}H \cap\IM^{-1}K$.

 Let $I,\ J\in \calId(\boldt)$. Then $I+J$ and $I\cap J\in \calId(\boldt)$, $\IM (I+J)=\IM I +\IM J,$ and 
  $\IM(I\cap J)=\IM(I) \cap\IM(J)$.

\medskip
 (2) (a).  Since $ \IM\,\IM^{-1}\,(H)$ is the image of all endomorphisms which map $G$ into $H,\   \IM \,\IM^{-1}\,(H)\leq H$. 
 
 Since  $ \IM^{-1}\,\IM \, (I)$ is the ideal of endomorphisms which map into the image of $I,\  
  \IM^{-1}\,\IM  (I) \leq I$.
 
 (b) By (a),  $\IM \,\IM^{-1}\,(H)\leq H$ and since $\IM^{-1}$ preserves order, $\IM^{-1}\IM \,\IM^{-1}\,(H)\leq \IM^{-1}H$. But by (a), $\IM^{-1}H\leq \IM^{-1}\IM  \IM^{-1}(H)$. 
 
 The proof for  $\calId(\boldt)$ is similar.
   \end{proof}

\begin{notation}Let $H\in \boldt$-inv$(G)$ and $I\in \calId(\boldt)$ . We say that 
\begin{enumerate}\item $H$  is
\textit{$\IM$--closed} if $H=\IM\,\IM^{-1}H $  and $I$ is
\textit{$\IM$--closed} if  $I=\IM^{-1}\IM( I)$.
 \end{enumerate}
The correspondence $\IM\colon \boldt$-inv$(G) \longleftrightarrow \calId(\boldt)$ defined in Proposition \ref{Gal} (1) is called the \textit{$\IM$--Galois Connection}.
 \end{notation}
\begin{proposition}\label{wesay} \begin{enumerate}\item Let $H\in \boldt$-inv$(G)$. The following are equivalent:
\begin{enumerate}\item there exists $I\in\calId(\boldt)$  such that    $H=\IM^{-1}(I)$;

\item   $H$ is $\IM$--closed;

\item there exists   a unique  $\IM$--closed   $I\in \calId(\boldt)$  such that $H=\IM^{-1}(I)$.
\end{enumerate} 
 
 item Let $I\in \calId(\boldt)$. The following are equivalent:
\begin{enumerate}\item there exists $H\in \boldt$-inv$(G)$ such that  $I= \IM(H))$;

\item   $I$ is $\IM$--closed;
\item there exists   a unique   $\IM$--closed    $H\in \boldt$-inv$(G)$ such that $I=  \IM(H)$.
\end{enumerate}
 \end{enumerate}
\end{proposition}

\begin{proof} (1)  $(a) \Rightarrow (b)$ Take $I=\IM^{-1} (H)$. Then $\IM\,\IM^{-1}(H) =\IM(I)=H$.

$(b) \Rightarrow (c)$ Once again, take $I=\IM^{-1}(H)$. Then $\IM(I)=\IM^{-1}\IM(H)=H$, so $\IM^{-1}\IM (I)=\IM^{-1}(H)=I$.
Suppose also $H=\IM^{-1}(J)$ with $J$ closed. Then $J=\IM^{-1}\,\IM ( J)=\IM^{-1}(H)= I$.

$(c) \Rightarrow (a)$ is clear.

(2) The proof is similar, \textit{mutatis mutandis}, to the proof of (1).
  \end{proof}
 \begin{lemma}\label{closur} Let $H\in \boldt$-inv$(G)$ and $I\in \calId(\boldt)$. 
  $H$  is $\IM$--closed  if and only if $\IM^{-1}(H)$ is $\IM$--closed, and $I$ is 
  is $\IM$--closed  if and only if $\IM(I)$ is $\IM$--closed.
\end{lemma}

\begin{proof} (1)  $H$  is $\IM$--closed if and only if $H=\IM \IM^{-1}(H)$,
  if and only if $\IM^{-1} H=\IM^{-1}\IM\IM^{-1}(H)$  if and only if $\IM^{-1}(H)$ is $\IM$--closed. 
 
 (2) The proof is analogous. \end{proof}   

\begin{corollary}\label{max} Each $H\in\boldt$-inv$(G)$ is contained in a unique (up to isomorphism)   $\IM$--closed subgroup $\IM\IM^{-1}(H)$ in $\boldt$-inv$(G)$ 

Each $I\in \calId(\boldt)$ is contained in a unique (up to isomorphism)   $\IM$--closed ideal $\IM^{-1}\IM(I)$ in $\calId(\boldt)$
\qed\end{corollary}
\subsection{The closed groups and ideals}

I show now that for all $\boldsigma\in\calI^{<\omega}(G),\ G(\boldsigma)$ is an $\IM$-closed subgroup and for all $\boldsigma\in\calI (\boldt),\ \boldt(\boldsigma)$ is an $\IM$--closed ideal.  It will become clear later that every for every $\IM$-closed subgroup  
$H$, there exists  $\boldsigma\in\calI^{<\omega}(G)$ such that $H= G(\boldsigma)$, and for every $\IM$-closed ideal  
$I$, there exists  $\boldsigma\in\calI (\boldt)$ such that $I= \boldt(\boldsigma)$.
\begin{lemma}\label{admissib} Let $G$ and $\boldt$ be as above. Then $\boldsigma\in \calI^{<\omega}(G)$ if and only if $\boldsigma\in\calI(\boldt)$.
\end{lemma}

\begin{proof} Since $G$ has a cyclic summand of exponent $n$ if and only if $\boldt$ has a cyclic summand of exponent $n,\ \boldsigma$ is $G$--admissible if and only if $\boldsigma$ is $\boldt$--admissible.
\end{proof} 

 \begin{proposition}\label{sufficient} Let $G$ be a group and let  $\boldsigma\in\calI^{<\omega}(G)$. Then \begin{enumerate}\item $\IM(\boldt(\boldsigma))=G(\boldsigma)$;
\item $\IM^{-1}(G(\boldsigma))=\boldt(\boldsigma)$.
\end{enumerate}\end{proposition}

\begin{proof}  (1) Let $f\in \boldt(\boldsigma)$. Since $f$ does not decrease height or increase exponent,   $Gf\leq G(\boldsigma)$. Conversely, let $a\in G(\boldsigma)$. Since $G$ has summands $\la b\ra$ of exponent $\geq\exp(a)$,  there exists $f\in\boldt(\boldsigma)$ such that $bf=a$. Hence $G(\boldsigma)\leq\IM(\boldt(\sigma))$.

(2) By (1), $\boldt(\boldsigma)\leq \IM^{-1} G(\boldsigma)$. Conversely, let $f\in \IM^{-1} G(\boldsigma)$.
Since Since $G$ has summands $\la b\ra$ of arbitrarily high exponent such that $bf\in G(\boldsigma)$, $f\in \boldt(\boldsigma)$.
\end{proof}
\begin{corollary}\label{corro} For all $\boldsigma\in\calI^{<\omega},\  G(\boldsigma)$ and $\boldt(\boldsigma)$ are $\IM$--closed.
\end{corollary}
\begin{proof}
Since $G(\boldsigma)=\Im(\boldt(\boldsigma))$ Proposition \ref{wesay} implies that $G(\boldsigma)$ is $\IM$--closed. The proof for ideals is similar:
\end{proof}

Recall from Proposition \ref{4.11} that if $H\leq G$ then $H$ is $\boldt$--invariant if and only if there exists $\boldsigma$ such that $H=G(\boldsigma)$ and from Corollary \ref{corro} that $H$ is $\IM$--closed. The situation for ideals is different: 

\begin{lemma}\label{except} Let   $\boldsigma$ be a $\boldt$--admissible indicator. Then $ I=\{f\in \boldt(\boldsigma)\colon \rank(Gf)< \aleph_0\}\in\calId(\boldt)$ with $\IM(I)=G(\boldsigma)$, while $\IM^{-1}G(\boldsigma)=\boldt(\boldsigma)$..
\end{lemma} 
\begin{proof}
Since $I$ is closed under addition and left and right multiplication from $\boldt,\ I\in\calId(\boldt)$.

Let  $a\in G(\boldsigma)$ have exponent $n$. Since $G$ has a cyclic summand $\la b\ra$ of exponent $\geq n,\ 
\boldt$ has an element $f$ with  $\rank(Gf)=  1$ such that $bf=a$. 
\end{proof} 

\begin{corollary}\label{and} Let   $\boldsigma$ be a $\boldt$--admissible indicator and $\boldm=(m_i\colon i\in\bbN)$ a non--decreasing sequence of infinite cardinals $\leq \rank(G)$. Then $ I_i=\{f\in \boldt(\boldsigma)\colon \rank(Gf)< m_i\}\in\calId(\boldt)$ with $\IM(I)=G(\boldsigma)$ and $(I_i)$ is a non--decreasing chain in $\calId(\boldt)$.
\qed\end{corollary}

\begin{proposition} \label{rela} Let  $I,\ J\in\calId(\boldt)$.  \begin{enumerate}\item $\IM(I)=\IM(J)$ is an equivalence relation on   $\calId(\boldt) $;
 \item Each equivalence class   contains a unique $\IM$--closed ideal $\boldt(\boldsigma)$ which is maximal in the class with respect to inclusion.  \end{enumerate}
 \end{proposition} 
  \begin{proof}
 (1) The definition $I\equiv J$ if and only if $\IM(I) =\IM(J)$ shows that the relation $\equiv$ is reflexive, symmetric and transitive.
 
 (2) We have seen in Proposition \ref{wesay} (2) that  $I\in\calId(\boldt)$ is closed if   $I=\{f\in\boldt\colon Gf\leq \IM^{-1}I\}$. Since $\IM^{-1}I$ is $\boldt$--invariant, by Proposition \ref{4.11}, there exists $\boldsigma\in\calI(\boldt)$ such that $\IM^{-1}I= G(\boldsigma)$.

Hence   $\boldt(\boldsigma)$
is the unique maximal ideal in $[I]=[\boldt(\boldsigma)]$.
 \end{proof}
\section{The classification of  $\calId(\boldt)$.}

We have seen that the lattice $ \calI(\boldt)$ of $\boldt$--admissible indicators determines a partition $\{[\boldt(\boldsigma)]\}$ of $\calId(\boldt)$. Consequently,
 a description of  each class $[\boldt(\boldsigma)]$ will complete the classification of $\calId(\boldt)$.
 
  \begin{notation} Recall that $G$ is an unbounded group, $\boldt$ is the torsion ideal of $\calE(G)$, $\calId(\boldt)$ the lattice of ideals of $\boldt$ and $\calI(\boldt)$ the lattice of $\boldt$--admissible indicators.
  
  \begin{enumerate}\item  For any     $\boldsigma=(\sigma_i\colon i\in\bbN)\in \calI(\boldt)$
  and    $j\in\bbN$, let $\boldsigma^{(j)}=(\sigma_0,\dots,\sigma_j,\infty)$ and $\boldsigma^{(\omega)}=\boldsigma$. Note that each $\boldsigma^{(j)}\in \calI(\boldt)$ and $\boldsigma^{(j+1)}\preceq\boldsigma^{(j)}$.

  \item Let $\boldsigma\in \calI (\boldt)$,   and let $\boldm=(m_{i}\colon i\in\bbN\,\ m_i\leq \rank(G))$ be a non--increasing sequence of infinite cardinals.  Note that although nominally infinite,    the sequence $\boldm$ contains only finitely many distinct terms. If $m_n$ is the maximum term, we write $\boldm=(m_0,\dots, m_n)$. 
   \item    The sequence $(\boldsigma,\boldm)=(\sigma_i,\,m_{i}\colon i\in\bbN)$ is called a \textit{ranked $\boldt$--admissible indicator};
 \item         Ranked $\boldt$--admissible indicators are ordered     pointwise, i.e.,  $(\boldsigma,\boldm)\preceq (\boldtau,\boldn)$ if and only if for all $i\in\bbN,\ \sigma_i\leq \tau_i$ and $m_i\leq n_i$.
  For example,  $ (\boldsigma^{min}, (\aleph_0))$ is the minimum ranked indicator and  $((\infty), (rank(G))) $ the maximum, where $\boldsigma^{min}=(0,1,\dots)$.
 \end{enumerate}
 \end{notation}
  
  The proof of the following lemma is evident.
 \begin{lemma}\label{lat}\begin{enumerate}\item
  The set of ranked $\boldt$--admissible indicators is a lattice under the operations $(\boldsigma,\boldm)\wedge(\boldtau,\boldn)=(\boldsigma\wedge\boldtau,\boldm\wedge \boldn)$ and $(\boldsigma,\boldm)\vee(\boldtau,\boldn)=(\boldsigma\vee\boldtau,\boldm\vee \boldn)$.

\item Let $\boldm=(m_i\colon i\in\bbN)$ be a non--increasing sequence of infinite cardinals $\leq\rank(G)$. For each $j\in\bbN$, let $\boldm^{(j)}$ be the truncated sequence $(m_0,\dots, m_j)$ and let $(\boldsigma^{(j)}, \boldm^{(j)})$ be the corresponding truncated ranked $\boldt$--admissible indicator. 	
 Then $\left((\boldsigma^{(j)}, \boldm^{(j)})\colon j\in\bbN\right)$ is an increasing sequence of ranked indicators  whose supremum is $(\boldsigma,\boldm)$.
 \qed\end{enumerate}\end{lemma}  
  For any ranked   $\boldt$--admissible indicator $(\boldsigma,\,\boldm)$,
  and any 
 $j\leq\omega$, let $\boldt(\sigma^{(j)},m_{j})=\{f\in \boldt(\sigma^{(j)}) \colon \rank(Gf)< m_j\}$ and $\boldt(\boldsigma,\boldm)=\sum_{j\in\bbN}\boldt(\sigma^{(j)},m_{j})$.
    \begin{lemma}\label{4.9}  Let $(\boldsigma,\,\boldm)$ be a ranked $\boldt$--admissible indicator.      
   \begin{enumerate}\item for all $j\in\bbN, \ \boldt(\sigma^{(j)},m_{j})\in\calId(\boldt)$;

 \item  $\boldt(\boldsigma,\,\boldm)\in\calId(\boldt)$;
    \item $\boldt(\boldsigma,\boldm)\leq\boldt(\boldtau,\boldn)$ if and only if $\boldtau\preceq\boldsigma$ and   $\boldm\leq\boldn$.    
  
 \item   The poset of ranked indicators and the corresponding poset of ideals are lattices.
  \end{enumerate}
   \end{lemma}
   
   \begin{proof} (1) Let $f,\ g\in  \boldt(\sigma^{(j)},m_{j})$, and $h\in\boldt$. Then $f+g,\ hf$ and  $fh\in \boldt(\sigma^{(j)},m_{j})$, so $\boldt(\sigma^{(j)},m_{j})$ is an ideal of $\boldt$. 
   
(2) As a sum of ideals, $\boldt(\boldsigma,\boldm)\in\calId(\boldt)$.

(3) For all $j\in\bbN,\ \boldt(\boldsigma^{(j)},m_j)\leq \boldt(\boldtau^{(j)},n_j)$ if and only if $\boldtau^{(j)}
\preceq  \boldsigma^{(j)}$ and $m_j\leq n_j$. Hence by Lemma \ref{lat} (2), $\boldt(\boldsigma,\boldm)\leq\boldt(\boldtau,\boldn)$ if and only if $\boldtau\preceq\boldsigma$ and   $\boldm\leq\boldn$.  

 (4)  
Since $\boldn,\ \boldm$ have only finitely many distinct entries,  so do their pointwise minimum $\boldn\wedge \boldm$ and maximum $\boldn\vee \boldm$. 
 Defining  $(\boldsigma,\boldn)\wedge (\boldtau,\boldm)=(\boldsigma\vee\boldtau, \boldn\wedge \boldm)$ and $(\boldsigma,\boldtau)\vee (\boldtau,\boldm)=(\boldsigma\wedge\boldtau, \boldn\vee \boldm)$ yields the required lattice of ranked indicators.
 
 Let $\boldt(\boldsigma,\boldm)$ and $\boldt(\boldtau,\boldn)\in\calId(\boldt)$. By (3),  in the lattice $\calId(\boldt),\ \boldt(\boldsigma,\boldm)\wedge\boldt(\boldtau,\boldn)=\boldt(\boldsigma\vee\boldtau, \boldn\wedge \boldm)$ and $\boldt(\boldsigma,\boldm)\vee\boldt(\boldtau,\boldn)=\boldt(\boldsigma\wedge\boldtau, \boldn\vee \boldm)$.
  \end{proof}  
  
 I   now use ranked indicators to determine the equivalence classes of the relation $\IM$ of \S5. 

 \begin{lemma}\label{rais} Let $(\boldsigma,\boldm)$  be a ranked $\boldt$--admissible indicator. Then $\boldt(\boldsigma,\boldm)\in\calId(\boldt)$ is in the equivalence class $[\boldt(\boldsigma)]$.
 \end{lemma}

\begin{proof}    For all $j\ m_j\geq\aleph_0,\  \IM(\boldt(\sigma^{(j)},m_{j}))=G(\sigma^{(j)})$. Hence $\IM(\boldt(\boldsigma, \boldm))=G(\boldsigma)$ so
 $\boldt(\boldsigma, \boldm)\equiv \boldt(\boldsigma)$.\end{proof}
 
  I now show that conversely,   every ideal   in $[\boldt(\boldsigma)]$   is $\boldt(\boldsigma,\boldm)$ for some  sequence $\boldm$ of cardinals.
 \begin{proposition}\label{MA1} For every $I\in[\boldt(\boldsigma)]$ there is a   non--increasing sequence $\boldm=(m_i)$ of infinite cardinals such that 
$I=\boldt(\boldsigma, \boldm)$.
\end{proposition}
\begin{proof}   
  Let $I_0= I\cap   \boldt(\sigma_0)$, so $I_0$ is an ideal   of the ring   of linear transformations of $p^{\sigma_0}G[p]$. By the known structure of such ideals, \cite{H82},  there is an infinite cardinal $m_0$ such that $I_0=\boldt(\sigma_0, m_0)$. 
 
For all $n\in\bbN$, let $I_n=I\cap \boldt(\sigma^{(n)})$. Then $I_n$ is an ideal of  $\boldt$ contained in $\boldt(\sigma^{(n)})$ and $(I_n\colon n\in\bbN)$ is a non--increasing chain.

Assume that  $m_0\geq\cdots\geq m_n$ is a sequence of infinite cardinals such that for all $j\leq n,\ I_j=\boldt\left( \boldsigma^{(j)}, m_j\right)$. 
  
 Then   $ \boldt\left( \boldsigma^{(j+1)}, m_{j}\right)\leq I_j$ is an ideal of $\boldt$ such that 
  $I_j/\boldt\left( \boldsigma^{(j+1)}, m_{j}\right)$ is an $\bbF_p$--space containing the ideal  $I_{j+1}/\boldt\left( \boldsigma^{(j+1)}, m_{j}\right)$  
 Hence by Hausen's classification of ideals of a vector space cited above, there is an infinite cardinal $m_{j+1}\leq m_j$ such that  $I_{m+1}=  \boldt\left( \boldsigma^{(j+1)}, m_{j+1}\right)$.
  
 Let $\boldm=(m_j\colon j\in\bbN)$. By induction
 we have constructed a ranked indicator $(\boldsigma,\boldm)$ for which $I=\sum_{j\in\bbN} I_j=
 \boldt(\boldsigma,\boldm)$.
\end{proof}

We can now state and prove the Main Theorem.

   \begin{theorem}\label{t-ind} Let $G$ be an unbounded group and $\boldt$ the torsion ideal of $\calE(G)$. Let $\boldt(\calI)$ be the lattice of $\boldt$--admissible indicators and $M$   the lattice of non--increasing sequences of infinite cardinals $\leq\rank(G)$.
 
 For every ideal $I$ of $\boldt$,    there is a unique    ranked  $\boldt$--admissible indicator $(\boldsigma, \boldm)$  such that $I=\boldt(\boldsigma,\,\boldm)$. 
 
 The ideal lattice of $\boldt$ is determined by the lattice of   ranked  $\boldt$--admissible indicators 
 $\{(\boldsigma,\boldm)\colon \boldsigma\in \boldt(\calI),\ \boldm\in M\}$.
\end{theorem}
\begin{proof}    Let $I\in\calId(\boldt)$. By Proposition \ref{rela} there is a unique $\boldt$--admissible indicator $\boldsigma$ such that $I\in[\boldt(\boldsigma)]$.

B y Proposition \ref{MA1} there is a ranked $\boldt$--admissible indicator   $(\boldsigma,\,\boldm)$    such that $I=\boldt(\boldsigma,\,\boldm)$. 

By Lemma \ref{4.9} $\boldt(\boldsigma,\,\boldm)\leq \boldt(\boldtau,\,\boldn)$ if and only if $\boldtau\preceq\boldsigma$ and $\boldm\leq\boldn$.-
 \end{proof}


\begin{thebibliography}{99}

\bibitem
[\textbf{AbS}, 2004]{AbS}
{R, Abraham and   P. Schultz }, \textit{Aditive Galois Theory of Modules}, in Rings, Modules, Algebras and Abelian Goups, Vol. 236,  Lecture Notes in Pure and Applied Mathematics, Marcel Dekker Inc., (2004)


\bibitem [\textbf{AvS}, 2000]{AvS} M. A. Avino,\ Phill Schultz,\textit{The endomorphism ring of a bounded abelian $p$--group}, in: Abelian groups, rings  and modules, Contemp. Math., 273, Amer. math. Soc., (2001), 75--84
 

\bibitem
[\textbf{ASZ}, 2020]{ASZ}
{M. A. Avi\~no,\ Phill Schultz and Marcos Zyman}, \textit{The upper central series of the maximal normal $p$--subgroup of a group of automorphisms}, Journal of Group Theory, 24(6), 1213-1244, (2020)
\bibitem
[\textbf{Birkhoff}, 1967]{Birkhoff} G.. Birkhoff, \textit{Lattice Theory},  A.M.S. Colloquium Publ., Vol.26, Third Ed., (1967).
 .

 
\bibitem
[\textbf{Corner}, 1969]{Corn} A. L. S. Corner, \textit{On endomorphism rings of primary abelian groups}, Quart. J. Math. Vol. 20, No. 79 (1969), 277-=296

 

\bibitem[\textbf{Fuchs}, 2015]{Fuchs}
{ L.~Fuchs}, \textit{Abelian Groups}, Springer Monographs in Mathematics, Springer 2015. 


\bibitem[\textbf{Hausen}, 1982]{H82}
{J. Hausen}, \textit{ Infinite general linear groups over rings},
Archjv der Math., 39, (1982) 510--524.


\bibitem[\textbf {Krylov, Mikhalev and Tuganbaev, 2003}] {KMT}  P. . Krylov, A. V. Mikhalev and A. A. Tuganbaev, \textit{Endomorphism Rings of Abelian Groups} KluwernAcademic Publishers (2003).


\bibitem[\textbf{Pierce}, 1963]{Pierce}
{ R. S. Pierce}, \textit{Homomorphisms of primary abelian groups}, in \textbf{Topics in Abelian Groups}Scott, Foresman and Co., 1963, 215--310.

\bibitem[\textbf{Stringall, 1967}]{String}\textit{Endomorphism rings of primary abelian groups} 
Pac, J.  Math.   20,  . 3, (1967), 535--557.


\end{thebibliography}
\end{document}